\documentclass[a4paper,10pt,hidelinks]{amsart}
\usepackage{graphicx} 
\usepackage[a4paper, margin=2 cm]{geometry}
\usepackage{hyperref}
\usepackage{amsthm}
\usepackage{thm-restate}
\usepackage{geometry}
\usepackage{amssymb}
\usepackage{amsmath}
\usepackage{cleveref}
\usepackage{comment}
\usepackage{setspace}
\usepackage{xcolor}
\usepackage[shortlabels]{enumitem}
\usepackage{listings}
\usepackage{blkarray}
\usepackage{lmodern}
\usepackage[english]{babel}
\usepackage{float}
\usepackage{tikz}
\usetikzlibrary{arrows.meta,arrows}
\usetikzlibrary{arrows,decorations.markings}
\usetikzlibrary{decorations.pathmorphing}
\usepackage[square,sort,comma,numbers]{natbib}

\newtheorem{theorem}{Theorem}[section]
\newtheorem{lemma}[theorem]{Lemma}

\newtheorem*{claim*}{Claim}

\newtheorem{conjecture}[theorem]{Conjecture}

\theoremstyle{definition}

\newtheorem*{definition*}{Definition}

\newcommand{\U}{\mathcal{U}}
\newcommand{\Y}{\mathcal{Y}}

\newcommand{\eps}{\varepsilon}

\newcommand{\Bin}{\mathrm{Bin}}

\renewcommand{\Pr}{\mathbb{P}}

\newcommand{\ceil}[1]{
    \left\lceil #1 \right\rceil
}
\newcommand{\floor}[1]{
    \left\lfloor #1 \right\rfloor
}

\title{Topological Minors in Typical Lifts}

\author
{
Matija Buci\'{c}
}
\author 
{
Micha Christoph
}
\author
{
Alp Müyesser 
}
\thanks{The third author was supported by the Cecil King Travel Scholarship of the Cecil King Foundation and the London Mathematical Society. The second and fourth authors were supported by the Swiss National Science Foundation Ambizione Grant No. 216071.}
\author
{
Raphael Steiner
}

\address[Buci\'{c}]{School of Mathematics, Institute for Advanced Study and Department of Mathematics, Princeton University, Princeton, USA.}

\email{\tt mb5225@princeton.edu}

\address[Müyesser]{Department of Mathematics, University College London, UK.}

\email{\tt alp.muyesser.21@ucl.ac.uk}

\address[Christoph, Steiner]{Department of Computer Science, Institute of Theoretical Computer Science, ETH Z\"{u}rich, Switzerland.}
\email{\tt $\{$micha.christoph, raphaelmario.steiner$\}$@inf.ethz.ch}

\date{}

\begin{document}

\begin{abstract}
    \setlength{\parskip}{\smallskipamount}
    \setlength{\parindent}{0pt}
    \noindent
    An $\ell$-lift of a graph $G$ is any graph obtained by replacing every vertex of $G$ with an independent set of size $\ell$, and connecting every pair of two such independent sets that correspond to an edge in $G$ by a matching of size $\ell$. Graph lifts have found numerous interesting applications and connections to a variety of areas over the years. Of particular importance is the random graph model obtained by considering an $\ell$-lift of a graph sampled uniformly at random. This model was first introduced by Amit and Linial in 1999, and has been extensively investigated since. In this paper, we study the size of the largest  topological clique in random lifts of complete graphs. 
    
   \par  In 2006, Drier and Linial raised the conjecture that almost all $\ell$-lifts of the complete graph on $n$ vertices contain a subdivision of a clique of order $\Omega(n)$ as a subgraph provided $\ell$ is at least linear in $n$. We confirm their conjecture in a strong form by showing that for $\ell \ge (1+o(1))n$, one can almost surely find a subdivision of a clique of order~$n$. We prove that this is tight by showing that for $\ell \le (1-o(1))n$, almost all $\ell$-lifts do not contain subdivisions of cliques of order $n$.
   
   \par Finally, for $2 \le \ell \ll n$, we show that almost all $\ell$-lifts of $K_n$ contain a subdivision of a clique on $(1-o(1))\sqrt{\frac{2n \ell}{1-1/\ell}}$ vertices and that this is tight up to the lower order term.
 \end{abstract}

\maketitle

\section{Introduction}
Instances of the following classical meta-question are encountered frequently across combinatorics. Given a graph $G$, how can we generate a larger graph which inherits to some extent the structure of $G$? Perhaps the most natural way is to replace each vertex $v$ of $G$ with a collection of $\ell$ vertices $G_v$ and then ``join'', for any edge $uv$ of $G$, the sets $G_u$ and $G_v$. There is a variety of ways in which the ``join'' operation can be performed. For example, one can place a complete bipartite graph between $G_u$ and $G_v$ for every adjacent $u$ and $v$, in which case one obtains the classical notion of an $\ell$-blow-up of $G$. While blow-ups proved very useful for a variety of problems over the years they have a downside in that they substantially increase the degree as we take $\ell$ to grow, making them less suitable for constructing sparser graphs. On the other side of the spectrum, one might instead place a perfect matching between any $G_u$ and $G_v$ for every edge $uv$ of $G$. This recovers another classical notion, that of an \emph{$\ell$-lift} of a graph.

The initial motivation for the study of lifts of graphs came from topology. The reason behind this is that an equivalent way of defining a lift of a graph $G$ is as any graph that can be mapped to $G$ via a covering map, a well-studied and quite widely useful object in topology. Since in this paper we focus on the graph theoretic and probabilistic point of view, we point an interested reader to \cite{linial1999} for precise definitions and more details on this connection. One piece of notation we do inherit from the topological point of view is in that we refer to the set $G_v$ as the \emph{fiber} of $v$.

Lifts have found numerous interesting applications across mathematics and computer science. For some examples in theoretical computer science, coding theory, cryptography, quantum information theory, and distributed computing see e.g.\ \cite{bilu-linial, eigenvalues, distributed,unique-games, quantum,coding, expander-survey,brailovskaya2022universality,near-ramanujan}. Let us highlight that they were instrumental in the recent celebrated work of Marcus, Spielman, and Srivastava \cite{interlacing} showing the existence of bipartite Ramanujan graphs of any degree.

A key question that arises is how should one choose the matchings between fibers. One of the most common (and used in many aforementioned applications) ways is to choose the matching uniformly at random between any pair of fibers corresponding to adjacent vertices. This gives rise to a very natural and well-studied model for generating random graphs. It was first explicitly introduced by Amit and Linial in 1999 \cite{linial1999}. One of their main motivations was the fact that, compared to the usual binomial model of random graphs, one gets much more control over the structure of the random graph we are generating. 
\par We note that the binomial random graph model found numerous applications in extremal combinatorics in large part due to its simplicity which eases its analysis. However, the simplicity behind the binomial random graph introduces certain limitations. The random lift model allows more flexibility while still being relatively simple to analyze. There has been plenty of work towards understanding the properties of random graphs generated this way 
(see \cite{linial1999,linial2006minors,eigenvalues,edge-expansion,alpha-chi,PMs,spectrum-random-lifts,HCs, HCs2,bernshteyn2023dp,nir2021chromatic,witkowski2013,PMs2} for just some examples) over the last 25 years. Investigating the properties of random lifts of complete graphs is an important instance which attracted a lot of attention over the years, 
not least because such random lifts correspond in some sense most closely to a binomial random graph. On the other hand, the two models often exhibit quite different behavior, and there remain several fundamental open problems regarding the typical properties of lifts of complete graphs. Following \cite{linial1999}, we say that almost all $\ell$-lifts of $K_n$ satisfy a property if all but a vanishing proportion of them do, as $n\to \infty$.

In this paper, we study the size of the largest topological cliques one can typically find in an $\ell$-lift of $K_n$. Here, a \emph{topological clique} of order $m$ is a graph consisting of $m$ \emph{branch} vertices all pairs of which are joined by internally vertex-disjoint paths. The largest order of a topological clique one can find as a subgraph of a graph $G$ is called its \emph{Haj\'os number}. The Haj\'os number is a well-studied graph parameter with a long history dating back to the 1940's (see \cite{paul-survey}) when Haj\'os made his famous conjecture that the Haj\'os number of a graph is always at least as large as its chromatic number. This conjecture has been disproved by Catlin \cite{catlin} in 1979. That the difference can be quite substantial was shown by Erd\H{o}s and Fajtlowicz \cite{Hajos-random-graphs} who proved that with high probability the binomial random graph $\mathcal{G}(n,1/2)$ has Haj\'os number $\Theta(\sqrt{n})$ (as opposed to chromatic number which is well-known to be $\Theta(n/\log n)$, see e.g.\ \cite{chi-Gnp}). Bollob\'as and Catlin \cite{bollobas-catlin} even determined the correct leading constant and extended the result to arbitrary constant density $p$. In 1979 Ajtai, Koml\'os, and Szemer\'edi \cite{ajtai-komlos-szemeredi} determined the behavior even when $p \to 0$. 

We note that the Haj\'os number has been investigated for a number of other classes of graphs besides random ones. Perhaps the most spectacular result in this direction is one due to K\"uhn and Osthus \cite{kuhn-osthus-girth} from 2002 showing that any graph with minimum degree $d$ and girth at least $186$ has Haj\'os number equal to $d+1$ (see also \cite{improved-kuhn-osthus} for an improved bound on the girth). Turning to more recent results, Liu and Montgomery \cite{mader-resolution} show that for any $s,t \ge 2$ any $K_{s,t}$-free graph of average degree $d$ will contain a topological clique of order $\Omega(d^{\frac12 \cdot \frac{s}{s-1}})$. Already the case $s=t=2$ of their result settled a well-known conjecture of Mader \cite{mader} and showed that $C_4$-free graphs have Haj\'os number of at least $\Omega(d)$. Let us also highlight a very nice recent work of Dragani\'{c}, Krivelevich, and Nenadov \cite{draganic2022rolling} on finding large topological cliques in expander graphs. 
We point an interested reader to \cite{hajos-summary,crux} for excellent summaries of what is known about the Haj\'os number and to \cite{thomassen} for some interesting connections to other topics. 

In the setting of random lifts, the problem of determining the Haj\'os number was first considered by Drier and Linial \cite{linial2006minors} in 2006.
They showed that almost every $\ell$-lift of $K_n$ has a topological clique of order at least $(1-o(1))\frac{\ell}3$ provided $\ell < (1-o(1))\frac{n}{2}$. Combined with the immediate upper bound of $n$ (maximum degree in a topological clique can not exceed that of a graph containing it), this tells us that when $\ell$ is linear in $n$ and at most about $n/2$, the Haj\'os number of almost all $\ell$-lifts of $K_n$ is $\Theta(n)$. Since their result does not inform the case when $\ell \ge n/2$, Drier and Linial naturally asked what happens for larger $\ell$. In particular, they posed the following conjecture in this direction.

\begin{conjecture}[\cite{linial2006minors}]\label{conj:main}
    For $\ell \ge \Omega(n)$ almost all $\ell$-lifts of $K_n$ have Haj\'os number equal to $\Theta(n)$.
\end{conjecture}

In 2013, Witkowski \cite{witkowski2013} proved that for $\ell$ large enough the conjecture is true. In fact, he proved the Haj\'os number is precisely $n$ for sufficiently large $\ell$. His argument, however, requires $\ell$ to be exponential in $n$. At a cost of precision, several of the aforementioned recent results on the Haj\'os number of various classes of graphs can be used to reduce the requirement on the size of $\ell$. Combining well-known results on the spectrum of random lifts, establishing the fact that they are excellent expanders (see e.g.\ \cite{spectrum-random-lifts}), with the work of Dragani\'{c}, Krivelevich, and Nenadov \cite{draganic2022rolling} on finding large topological cliques in expander graphs (see \Cref{sec:conc-remarks} for more details on this result) one can prove that \Cref{conj:main} holds if $\ell \ge n^{O(1)},$ and they in fact find topological cliques of order $(1-o(1))n$. Using the result of Liu and Montgomery \cite{mader-resolution} on $C_4$-free graphs\footnote{A typical lift in the relevant regime here will not be $C_4$-free but will only have a few $C_4$'s which can be removed by a standard application of the alteration method without hurting the average degree too much, provided $\ell\gg n^2$.}
one can prove the conjecture already for $\ell \ge O(n^2),$ although their arguments provide only a topological clique of order $\Omega(n)$. \par We completely settle the above conjecture of Drier and Linial in a very strong form.

\begin{theorem}\label{thm:maintheorem-intro}
For $\ell\geq (1+o(1))n$ almost all $\ell$-lifts of $K_n$ have Haj\'os number equal to $n$.
\end{theorem}
We note that besides clearly settling \Cref{conj:main} for $\ell \ge (1+o(1))n$, Theorem~\ref{thm:maintheorem-intro} also settles \Cref{conj:main} for any $\ell\ge \Omega(n)$ as one can simply restrict attention to only a subset of $(1-o(1))\ell$ of the fibers and by \Cref{thm:maintheorem-intro} find a topological clique of order $(1-o(1))\ell\ge \Omega(n)$ among these fibers.

\par Furthermore, \Cref{thm:maintheorem-intro} actually determines the ``threshold'' in terms of how large $\ell$ we need to take for the Haj\'os number of a typical $\ell$-lift of $K_n$ to become equal to $n$.

\begin{theorem}\label{thm:lower_bound-intro}
    For $\ell\leq (1-o(1))n$ almost all $\ell$-lifts of $K_n$ do not contain a topological clique of order $n$.
\end{theorem}

These results settle the behavior of the Haj\'os number of a typical $\ell$-lift of $K_n$ for $\ell\ge \Omega(n)$. It is natural to ask what happens for smaller values of $\ell$. Here the best-known lower bound for very small values of $\ell$ was $\Omega(\sqrt{n})$ by the classical result of Koml\'os and Szemer\'edi \cite{komlos-szemeredi} and Bollob\'as and Thomason \cite{bollobas-thomason} which guarantees a topological clique of order $\Omega(\sqrt{d})$ in any graph with average degree at least $d$ (see also \cite{crux} for a conditional improvement in terms of a ``crux'' of a graph). For somewhat larger $\ell \le \frac n2$, the above mentioned lower bound of $\Omega(\ell)$  starts to dominate. On the other hand, the best known upper bound was $O(\sqrt{n\ell})$ due to Drier and Linial \cite{linial2006minors} showing that at the two extremes, when $\ell$ is a constant and when $\ell$ is linear both these lower bounds can be tight up to a constant factor. We determine the answer up to lower-order terms for any sublinear $\ell$.

\begin{theorem}\label{thm:smalll}
    For $2\le \ell \ll n$, almost all $\ell$-lifts of $K_n$ have Haj\'os number equal to $(1+o(1))\sqrt{\frac{2n\ell}{1-1/\ell}}$.
\end{theorem}

To prove Theorem~\ref{thm:maintheorem-intro} and Theorem~\ref{thm:smalll}, we rely on expansion properties of random lifts of $K_n$ (which are established in Section~\ref{sec:expansion}), and the so-called ``extendability method'' to embed vertex-disjoint paths between branch vertices. We give a more detailed overview in Section~\ref{sec:uppr-bnd}. The lower bound part of \Cref{thm:smalll} is proved in \Cref{sec:smalll}. The upper bound part of \Cref{thm:smalll} and \Cref{thm:lower_bound-intro} is further discussed in Section~\ref{sec:lwr-bnd}.

\textbf{Notation.} 
Given a graph $G$ we denote its vertex set by $V(G)$ and its edge set by $E(G)$. We write $e(G)$ for $|E(G)|$. $\Delta(G)$ stands for the maximum degree of $G$. Given $v \in V(G)$ we denote by $d_G(v)$ the degree of $v$ in $G$. Given $U \subseteq V(G)$ we denote by $N_G(U)$ the neighborhood of $U$, namely the set of vertices in $U$ or with a neighbor in $U$. By $N'_G(U)$, we denote $\cup_{v \in U} N_G(u)$. Given two graphs $G$ and $H$, their union $G \cup H$ is a graph with $V(G \cup H)=V(G) \cup V(H)$ and $E(G \cup H)=E(G) \cup E(H)$. We refer to any path from $u$ to $v$ as a $uv$-path. 
$\Bin(n,p)$ denotes the binomial distribution with parameters $n$ and $p$.

We will denote by $\U^{(\ell)}(G)$ a uniformly at random sampled $\ell$-lift of $G$. Note that this is equivalent to sampling every perfect matching between fibers corresponding to adjacent vertices in $G$ uniformly at random (independently between distinct pairs of fibers).

\section{Preliminaries}\label{sec:preliminaries}
In this section, we provide two simple probabilistic statements.
\begin{lemma}\label{lem: randomness}
    Let $F$ be a bipartite graph with bipartition $A_1 \sqcup A_2,$ where $|A_1|=|A_2|=\ell$. Let $M$ be a uniformly random perfect matching between $A_1$ and $A_2$. Then,
    $$
    \Pr[M \cap E(F) = \emptyset]\leq e^{-\frac{e(F)}{2\ell}}.
    $$
\end{lemma}
\begin{proof}
    We prove the statement by induction on $\ell$. For $\ell=1$ the statement clearly holds, so let us assume $\ell \ge 2$. Let $v\in V(F)$ be a vertex with maximum degree in $F$ and suppose $v\in A_i$.
    Let $M(v)\in A_{3-i}$ be the (random) vertex to which $v$ is matched by $M$. Then,
    $$
    \Pr[vM(v)\notin E(F)] = 1-\frac{\Delta(F)}{\ell}\leq e^{-\frac{\Delta(F)}{\ell}}.
    $$
    Given any $u\in A_{3-i}$, if we condition on the outcome $M(v)=u$, we have that $M\setminus uv$ is a uniformly random perfect matching between  $A_i\setminus \{v\}$ and $A_{3-i}\setminus \{u\}$. By induction, we get that
    $$
    \Pr[(M\setminus uv)\cap E(F-\{u,v\})=\emptyset \mid M(v)=u]\le e^{-\frac{e(F-\{u,v\})}{2(\ell-1)}} \leq e^{-\frac{e(F)-2\Delta(F)}{2\ell}},
    $$
    where we used $e(F-\{u,v\})\geq e(F)-2\Delta(F)$.
    Since the bound is independent of $u$, if we put the two bounds together, we get 
    $$
    \Pr[M \cap F =\emptyset]\leq e^{-\frac{e(F)}{2\ell}}.
    $$
\end{proof}

We will also need the following basic anti-concentration lemma.
\begin{lemma}\label{lem: constant prob}
    There exists $\eta>0$ such that the following is true for $n\ge 2$. Let $X=X_1+\ldots+X_{n}$ be a sum of independent Bernoulli random variables such that for every $i$, $\Pr[X_i=1]\geq \frac{1}{96 n}$. Then, 
    $$
        \Pr[X\geq 2]\geq \eta.
    $$
\end{lemma}
\begin{proof}
     Consider the random variable $Y=Y_1+\cdots+Y_{n}$ where each $Y_i$ is a Bernoulli random variable independent of the other variables, with $\Pr[Y_i=1]=\frac{1}{96 n}$. Then, we may couple each $Y_i$ with $X_i$ such that $X_i\geq Y_i$ and it follows that 
    $$
        \Pr[X\leq 1]\leq \Pr[Y\leq 1]= \left(1-\frac{1}{96 n}\right)^{n}+n\cdot \frac{1}{96 n}\left(1-\frac{1}{96 n}\right)^{n-1}\leq e^{-1/96}+\frac{e^{-\frac{n-1}{96 n}}}{96}.
    $$
    The final expression is decreasing in $n$ so we can set $\eta:=1-(e^{-1/96}+\frac{e^{-1/192}}{96})>0$.
\end{proof}

\section{Expansion and pseudorandom properties}\label{sec:expansion}
In this section, we collect a number of basic pseudorandomness and expansion properties, which hold with high probability in our random lifts. 

A graph $G$ is said to be \emph{$m$-joined} if, for any two disjoint subsets $A,B\subseteq V(G)$ of size at least $m$ each there exists at least one edge in $G$ between $A$ and $B$. The first simple lemma shows our random lifts are strongly joined.

\begin{lemma}\label{lem: joined}
    Let $G \sim \U(K^\ell_{n})$. Then with high probability as $n \to \infty$, $G$ is $5\ell\log n$-joined.
\end{lemma}
\begin{proof}
    Fix two disjoint arbitrary sets $A,B\subseteq V(G)$ of size $K:=\lceil 5 \ell \log n\rceil$ each. Let $A_1,\ldots,A_n$ be the intersections of $A$ with the fibers of $G$ and similarly, $B_1,\ldots, B_n$ the intersections of $B$ with the fibers of $G$. For each $1\leq i<j\leq n$, we define $F_{i,j} = (A_i\times B_j) \cup (A_j\times B_i)$.
    Since for all $i$ we have $|B_i|\leq \ell$, we obtain
    $$
    \sum_{i<j}|F_{i,j}|=|A| |B|-\sum_{i=1}^n |A_i||B_i|\geq |A||B|-\sum_{i=1}^n |A_i|\ell=|A|(|B|-\ell) \ge K(5\ell\log n-\ell)\geq 4K\ell\log n. 
    $$
    Note that each edge in $F=\bigcup_{i<j}F_{i,j}$ connects $A$ and $B$. As all the matchings between different pairs of fibers are sampled independently, we may apply \Cref{lem: randomness} to each pair of fibers separately. Collecting all the terms, we get that the probability that $G$ does not contain any edge of $F$ is at most 
    $$
    e^{-\frac{|F|}{2\ell}}\leq e^{-2K\log n}=n^{-2K}.
    $$

    The number of choices for $A$ and $B$ is at most
    $$
    \binom{\ell n}{K}^2\leq \left(\frac{en}{5\log n}\right)^{2K}=o(n^{2K}).
    $$
    The result then follows by a union bound over all such choices.
\end{proof}

The second lemma establishes a strong expansion condition into a fixed subset of vertices provided that it contains plenty of vertices from each fiber.
\begin{lemma}\label{lem:expansion2}
    Let $\eps>0$ and let $G\sim \U^{(\ell)}(K_n)$. Let $V\subseteq V(G)$ contain at least 
    $\max\{9\eps \ell,\ell-n\}$ vertices from each fiber of $G$. Then, with high probability as $n \to \infty$, for every $U \subseteq V(G)$ we have $$|N(U)\cap V| \ge \min\left\{{\eps n|U|}, {\eps^{6}
    \ell n}\right\}.$$ 
\end{lemma}
\begin{proof}
    We split the proof into two parts. The first part will establish with high probability a minimum degree bound towards $V$ in $G$ (and may be viewed as showing the desired bound holds when $|U|=1$). The basic idea for the second part is to apply a union bound over all (not too large) subsets $U \subseteq V(G)$ and over all subsets $X \subseteq V$, not too large compared to $U$, which should contain $N(U) \cap V$. The issue with this approach alone is that since $V$ might be only a small proportion of $V(G)$ this places only a relatively weak restriction on where $U$ can send its neighbors. To address this we also add to the union bound for each $u \in U$ a set of $\eps n$ fibers in which $u$ should have a neighbor in $V$. Now, for this $u$ and these fibers, the restriction of its neighborhood within $V$ to belong to $X$ becomes substantial and gives us the result.

    For the first part let $\mathcal{E}$ denote the event that every vertex $v\in V(G)$ has more than $4\eps n$ neighbors in $V$. For a fixed vertex $v\in V(G)$ the probability that it has a neighbor in $V$ belonging to any given fiber is at least $9\eps$. Moreover, these events are independent as we vary across the fibers. This implies that the random variable $N$ counting the number of neighbors of $v$ in $V$ dominates $\Bin(n-1,9\eps)$.  Then, $\mathbb{E}[N]\geq 9\eps (n-1)\ge 8\eps n$. By Chernoff's inequality, we get that $N\leq 4\eps n$ with probability at most $e^{-\eps n}.$ Therefore, by a union bound over all vertices, if $\ell\leq e^{\eps n/2}$, $\mathcal{E}$ occurs with high probability. Suppose $\ell\geq e^{\frac{\eps n}{2}}$. Then, the probability that $v$ has $2$ neighbors in $V(G)\backslash V$ in any two given fibers is at most $\frac{n^2}{\ell^2}$, since $V(G)\backslash V$ contains at most $n$ vertices per fiber. Now by a union bound over all pairs of distinct fibers the probability that $v$ has at least $2$ neighbors in $V(G)\backslash V$ is at most $\binom{n}{2}\frac{n^2}{\ell^2}\leq\frac{n^4}{\ell^2}$. Since $\ell\geq e^{\frac{\eps n}{2}}$, we again get by a union bound over all $\ell n$ vertices in $V(G)$ that $\mathcal{E}$ happens with high probability.

    Turning to the second part, given a subset $U \subseteq V(G)$, such that $2\le |U|\le \ceil{\eps^5\ell}=:m$, let $U_1,\ldots, U_n$ be its intersections with the fibers. Let $f:U \to \binom{[n]}{d}$ be an assignment of sets of fibers of size $d:=\ceil{4\eps n}$ to vertices in $U$.
    Let $X\subseteq V$ be a set of size $\floor{\eps n|U|}$, and let $X_1,\ldots,X_n$ be its intersections with the fibers. Then, we denote by $\mathcal{A}(U,f,X)$ the event that each vertex $u\in U$ has an edge to $X_j$ for each $j\in f(u)$.  

    \begin{claim*}
        For every choice of $U,f,X$ as above, $\Pr[\mathcal{A}(U,f,X)]\leq \left(\frac{|U|}{4\ell}\right)^{d|U|/2}.$
    \end{claim*}
    \begin{proof}
        For each $i,j\in [n]$, let us denote by $U_{i,j}$ the set of vertices $u\in U_i$ for which $j\in f(u)$, and set $u_{i,j}=|U_{i,j}|$.
        Let $x_i$ be the number of vertices in $X_i$ and without loss of generality, assume that $x_1\geq \ldots\geq x_n$. Define
        $$
        x:=\frac{x_1+\ldots+x_{d}}{d},
        $$
        and note that $x\leq |X|/d$.
        Let $G_{i,j}$ be the event that all neighbors of $U_{j,i}$ in fiber $i$ belong to $X_i$. Then, $\Pr(G_{i,j}) \le \left( \frac{x_i}{\ell}\right)^{u_{j,i}}$. Furthermore, since the event only depends on the matching chosen between fibers $i$ and $j$ we have that $G_{i,j}$ is independent of all other $G_{i',j'}$ except $G_{j,i}$. This means that the probability that all these events happen simultaneously is at most (using the AM-GM inequality):
        $$
        \prod_{i \neq j} \left( \frac{x_i}{\ell}\right)^{u_{j,i}/2}=\prod_{j=1}^n \prod_{u\in U_j}\prod_{i\in f(u)}\left(\frac{x_i}{\ell}\right)^{1/2}\leq \prod_{j=1}^n \prod_{u\in U_j}\prod_{i=1}^{d}\left(\frac{x_i}{\ell}\right)^{1/2}\leq\left(\frac{x}{\ell}\right)^{d|U|/2}.
        $$       
        The desired statement now follows since $x\leq |X|/d\le |U|/4$.
    \end{proof}
    Let $\mathcal{A}$ be the event that no $\mathcal{A}(U,f,X)$ happens. The number of choices for $U,f,X$ where $|U|=u\le m$ and $|X|=\floor{\eps nu}$ is at most
    $$
    \binom{n}{d}^{u}  \binom{\ell n}{|X|} \binom{\ell n}{u}\le
    \left( \frac{e}{4\eps}\right)^{ud}\cdot \left( \frac{e\ell}{\eps u}\right)^{\eps nu}\cdot \left( \frac{e\ell n}{u}\right)^{u}\le 
    \left( \frac{\ell }{\eps^5 u}\right)^{ud/4} \cdot \left( \frac{e\ell n}{u}\right)^{u}.
    $$ 
    By a union bound, using the above claim we now get
    $$\Pr\left[\mathcal{A}^c\right]\leq \sum_{u=1}^{m}
    \left(\frac{u}{4\ell }\right)^{ud/2}\cdot \left( \frac{\ell }{\eps^5 u}\right)^{ud/4} \cdot \left( \frac{e\ell n}{u}\right)^{u} 
    = \sum_{u=1}^{m} \left( \frac{u }{16\eps^5 \ell}\right)^{ud/4}\left( \frac{e\ell n}{u}\right)^{u}=\sum_{u=1}^{m}{\left( \frac{u }{16\eps^5 \ell}\right)^{u(d-4)/4} \left(\frac{en}{16\varepsilon^5}\right)^u}.$$
    Recalling that $2\le u \le  m=\ceil{\varepsilon^5 \ell}$ and $d\ge 4\eps n$, we can bound the first term of the above sum by $(1/8)^{u(d-4)/4}$ and the second term by $(1+o(1))^{u(d-4)/4}$. Altogether, we obtain $\mathbb{P}\left[\mathcal{A}^c\right]\le \sum_{u=1}^\infty{(1/8+o(1))}^{u(d-4)/4}\rightarrow 0$ for $n \rightarrow \infty$. This shows that $\mathcal{A}$ occurs with high probability, as desired.

    By a final union bound, we get that with high probability both $\mathcal{A}$ and $\mathcal{E}$ occur. Assuming this is the case, suppose $|N(U) \cap V|\le \floor{{\eps n|U|}}$, for some set $U\subseteq V(G)$ with $|U|\le m$. $\mathcal{E}$ implies $|U| \ge 2$ and that for every $u \in U$ we can fix a set $f(u)$ consisting of $d$ fibers containing a neighbor of $u$ belonging to $V$. Let $X$ be $N(U) \cap V$ padded with additional vertices if needed to ensure $|X|=\floor{\eps n|U|}$. Now for each $u \in U$ and $j \in f(u)$, we know that the neighbor of $u$ in fiber $j$ belongs to $V$, so belongs to $N(U) \cap V \subseteq X$. This implies $\mathcal{A}(U,f,X)$ happens, contradicting $\mathcal{A}$. Hence, $|N(U) \cap V|> {\eps n|U|}$ for every $U\subseteq V(G)$ with $|U|\le m$. For any larger set $U$, we can apply this inequality to any subset consisting of $m$ vertices to conclude $|N(U) \cap V|\ge {\eps nm}\ge {\eps^{6}n\ell}$.
\end{proof}

The final auxiliary lemma shows that given a suitable collection of disjoint subsets of vertices of our lift, we can with high probability find a matching containing an edge between almost all pairs of our subsets. For us, the subsets in the collection we work with will all be transversals. We say that a subset of vertices of an $\ell$-lift of a graph is a \emph{transversal} if it contains exactly one vertex per fiber. We say it is a \emph{partial transversal} if it contains at most one vertex per fiber.
\begin{lemma}\label{lem: contains matching}
    Let $\gamma>0$ and $G\sim\U^{(\ell)}(K_{n-1})$ with $n\leq \ell \le \frac{\gamma^3}{48}\cdot n^2$.  
    Let $V_1,\ldots,V_n$ be a collection of disjoint transversals of $G$. Then with high probability, $G$ contains a matching $M$ such that for all but $\gamma n^2$ pairs $1\leq i<j\leq n$, $M$ contains an edge of the form $uv$ with $u\in V_i$ and $v\in V_j$.
\end{lemma}
\begin{proof}
Let $M$ be a maximal matching in $G[V_1\cup\ldots\cup V_n]$ such that for each $1\leq i<j\leq n$ there is at most one edge $uv\in M$ with $u\in V_i$ and $v\in V_j$. Let $H$ be an auxiliary graph with $V(H)=[n]$, where $ij\in H$ if and only if $M$ does not contain an edge $uv$ with $u\in V_i$ and $v\in V_j$. For each $i\in [n]$, let $U_i\subseteq V_i$ be the set of vertices not covered by $M$. Note that $|U_i|=d_H(i)$.
The following claim shows that we can ensure with high probability that $H$ contains less than $\gamma n^2$ edges.

    \begin{claim*}
        The following holds with high probability. For all graphs $H$ with $V(H)=[n]$ and $e(H)\geq \gamma n^2$ and for every collection of sets $U_1,\dots, U_n$ such that for every $1\leq i\leq n$, $U_i\subseteq V_i$ is of size $d_H(i)$, there exists some $ij\in H$ for which there is an edge between $U_i$ and $U_j$ in $G$.
    \end{claim*}
   
    \begin{proof}
        To prove the claim, we prove that the statement holds for an arbitrary choice of $H$ and $U_1,\dots,U_n$ with high enough probability so that we can finish the proof with a union bound over all these choices.
    For now, let us fix some graph $H$ and a collection of sets $U_1,\dots,U_n$ as in the statement and we say that $(H,U_1,\dots,U_n)$ is \emph{bad} if it does not satisfy the required property. 
    Let $X\subseteq [n]$ be the set of vertices of $H$ with degree at least $\gamma n/2$. Then, $H[X]$ contains at least $\gamma n^2/2$ edges, since $\sum_{i\in [n]\setminus X}d_H(i)\leq \gamma n^2/2$. 
    Let $E$ denote the set of pairs $uv$ with $u\in U_i$ and $v\in U_j$ such that $ij\in H[X]$ and $u$ and $v$ belong to different fibers of $G$. On a high level, the set $E$ is the set of potential edges that we will consider. Indeed, $(H,U_1,\dots,U_n)$ is not bad when $G$ contains an edge of $E$. 
    Note that the number of pairs in $E$ between $U_i$ and $U_j$ with $ij \in H[X]$ is at least $|U_i||U_j|-|U_i|$ since both $U_i$ and $U_j$ are subsets of transversals so they intersect any given fiber in at most one vertex. So,
    \begin{align*}
        |E|       
        &\geq \sum_{ij\in H[X]} |U_i||U_j|-|U_i|
        =\sum_{ij\in H[X]} d_H(i)(d_H(j)-1) \ge \frac{\gamma n^2}2\cdot \frac{\gamma n}{2}  \left(\frac{\gamma n}{2}-1\right)\ge \frac{\gamma^3n^4}{12},
    \end{align*}
    where in the penultimate inequality we used the fact that $H[X]$ has at least $\gamma n^2 /2$ edges and that for each $i\in X$ by definition of $X$ we have $d_H(i)\geq \gamma n/2$.
    
    For each $uv\in K_{n-1}$, denote by $E_{uv}$ the set of edges in $E$ between the fibers corresponding to $u$ and $v$. We say that $E_{uv}$ is \textit{bad} if $G$ does not contain an edge of $E_{uv}$ and $E$ is bad if and only if all of the $E_{uv}$ are bad. By Lemma~\ref{lem: randomness}, we have that
    $$
    \Pr[E_{uv}\text{ bad}]\leq e^{-\frac{|E_{uv}|}{2\ell}}.
    $$
    Since all the matchings are independent, we get that 
    $$
    \Pr[E\text{ bad}] = \prod_{uv\in K_{n-1}}\Pr[E_{uv}\text{ bad}]\leq e^{-\frac{|E|}{2\ell}}\leq e^{-\frac{\gamma^3n^4}{24\ell}}\leq e^{-2n^2},
    $$
    where we use the assumed upper bound on $\ell.$
    Finally, we union bound over all choices of $H$ and $(U_1,\dots,U_n)$. Since $H$ is a graph on $[n]$, there are at most $2^{n^2}$ choices for $H$. For each $i\in [n]$, there are $2^n$ choices for $U_i$. So, we get that in total there are at most $2^{2n^2}$ choices for $(H,U_1,\dots,U_n)$, so by a union bound the desired statement fails with probability at most $(2/e)^{2n^2}=o(1)$, completing the proof of the claim.
    \end{proof}
    By maximality of $M$ the conclusion of the claim can not hold for our auxiliary graph $H$, so it must have less than $\gamma n^2$ edges, as desired.
\end{proof}

To conclude the section we gather the conclusions of the above lemmas, slightly adapted for how we will use them in the next section. To this end, it will be convenient to have the following definition, which we borrow from \cite{montgomery2019spanning}. Here, to follow \cite{montgomery2019spanning}, we use $N'_G(U):=\cup_{v \in U} N_G(u)$.

\begin{definition*}
Let $D,m\in\mathbb N$ with $D\ge 3$.
Let $G$ be a graph and let $S\subset G$ be a subgraph with $\Delta(S) \leq D$.
We say that $S$ is \emph{$(D,m)$-extendable}
if for all $U\subset V(G)$ with $1\le |U|\le 2m$ we have
    \begin{equation}\label{def:extendability}
        |N'_G(U)\setminus V(S)|\ge (D-1)|U|-\sum_{u\in U\cap V(S)}(d_S(u)-1).
    \end{equation}
\end{definition*}

\begin{lemma}\label{lem:winningconditions}
Let $\varepsilon,\gamma>0$ and let $\ell\geq (1+\varepsilon)n$. Set $(D,m):=(n^{0.99}, 5\ell\log n)$. Let $G\sim\U(K^\ell_{n-1})$ and let $V_1,\ldots,V_n$ be a collection of disjoint transversals of $G$. Then, the following hold with high probability as $n \to \infty$.
\begin{enumerate}[label = {{\emph{\textbf{A\arabic{enumi}}}}}]
    \item\label{joined} $G$ is $m$-joined. 
    \item\label{extendable} The empty graph on vertex-set $V_1\cup\ldots \cup V_n$ is $(D,m)$-extendable in $G$.
    \item\label{pseudorandomstars} If $\ell \le \frac{\gamma^3}{48}\cdot n^2$, then $G$ contains a matching $M$ such that for all but $\gamma n^2$ pairs $1\leq i<j\leq n$, $M$ contains exactly one edge of the form $uv$ with $u\in V_i$ and $v\in V_j$.
\end{enumerate}
\end{lemma}
\begin{proof}[Proof of \Cref{lem:winningconditions}]
    With high probability, $G$ simultaneously satisfies the conclusions of \Cref{lem: joined,lem: contains matching} as well as the conclusion of \Cref{lem:expansion2} applied with $V=V(G)\setminus\left(V_1\cup\ldots\cup V_n\right)$ and $\eps/10$. Suppose $G$ does indeed satisfy these conclusions. \ref{joined} follows directly from \Cref{lem: joined} and \ref{pseudorandomstars} follows directly from \Cref{lem: contains matching}, where we remove unnecessary edges. Let us prove that \ref{extendable} follows from \Cref{lem:expansion2}. Let $U\subseteq V(G)$ be an arbitrary subset with $|U|\leq 10\ell\log n$. We get that 
    $$|N'_G(U)\setminus V|\geq \min\left\{\frac{\eps n|U|}{10}, \frac{\eps^{6}
    \ell n}{10^6}\right\}-|U|\geq n^{0.99}|U|,$$
    for $n$ large enough, which implies the desired extendability property as in (\ref{def:extendability}).
\end{proof}

\section{Finding a subdivision in a lift}\label{sec:uppr-bnd}
In this section, we provide a proof of \Cref{thm:maintheorem-intro}. In addition to the expansion properties we established in the previous section, collected in \Cref{lem:winningconditions}, one final ingredient is the following simple but convenient ``well-linkedness'' lemma of Montgomery which builds on work of Glebov, Johanssen, and Krivelevich (see \cite{montgomery2019spanning} for a more extensive survey, and see \cite{draganic2022rolling, hyde2023spanning,draganic2024hamiltonicity} for other recent applications of this method). It is well-known (and easy to show) that expander graphs have at most logarithmic diameter, meaning that between any pair of vertices, there exists a short connecting path. However, in order to find our desired clique subdivisions, we will need to embed several short connecting paths between various branch vertices in a vertex-disjoint manner. The following lemma asserts the existence of short connecting paths whose removal does not damage the expansion conditions of the underlying graph, making the lemma suitable to be invoked repeatedly to embed several vertex-disjoint paths at once. 
\begin{lemma}[{\cite[Corollary 3.12]{montgomery2019spanning}}]\label{lemma:connecting}
Let $D,m\in\mathbb N$ with $D\ge 3$.
Let $G$ be an $m$-joined graph which contains a $(D,m)$-extendable subgraph $S$ with at most $|V(G)|-10Dm$ vertices. Suppose that $u$ and $v$ are two distinct vertices in $S$ with $d_S(u),d_S(v)\le D/2$. 
Then, there exists a 
$uv$-path $P$ in $G$ of length at most $3\ceil {\log(2m)/\log(D-1)}$ such that 
all internal vertices of $P$ lie outside of $S$ and 
$S\cup P$ is $(D,m)$-extendable.
\end{lemma}
Our process for finding a $K_n$-subdivision (for Theorem~\ref{thm:maintheorem-intro}) begins by embedding $n$ branch vertices with disjoint neighborhoods (of size $n-1$). The remaining task is to find appropriate connecting paths between the neighborhoods of the branch vertices. Given \ref{joined} and \ref{extendable} of \Cref{lem:winningconditions} (applied with $V_i$ being the neighborhood of the $i$th branch vertex), \Cref{lemma:connecting} can be used repeatedly to find a significant portion of these connecting paths. However, the shortest connecting path guaranteed by \Cref{lemma:connecting} has length at least $3$, and we are required to make at least about $n^2$ connections between neighbors of branch vertices. Say when $\ell\leq 2n$, there simply aren't enough vertices in an $\ell$-lift of $K_n$ to accommodate $n^2$ many vertex-disjoint paths (disjoint with the initial branch vertices and internally disjoint with their neighborhoods) of length at least $3$. This is where \ref{pseudorandomstars} comes into play. Our algorithm uses \ref{pseudorandomstars} in order to convert the initially embedded branch vertices and their disjoint neighborhoods into a subdivision of a nearly complete graph on $n$ vertices whilst using only edges (or connecting paths of length $1$) in order to join up the neighbors of the branch vertices. This initial step costs us no additional vertices and is crucial to lowering the number of connections that need to be made afterwards, making it feasible to finish the job by repeated applications of \Cref{lemma:connecting}.

We are now ready to prove the upper bound part of our main theorem (Theorem~\ref{thm:maintheorem-intro}). We first state a more precise version.
\begin{theorem}\label{thm:maintheorem}
Let $\varepsilon>0$ and $\ell\geq (1+\varepsilon)n$. Then, $G\sim \U^\ell(K_n)$ contains a subdivision of $K_n$ as a subgraph with high probability as $n \to \infty$.
\end{theorem}

\begin{proof}
 Let $W\subseteq V(G)$ be a fiber of $G$. We reveal all the edges incident to $W$. Let $V_1,\ldots, V_n$ be the neighborhoods of $n$ arbitrary but distinct vertices of $W$. Our plan (as explained in the paragraph after Lemma~\ref{lemma:connecting}) is to embed the branch vertices of our subdivision of $K_n$ into these $n$ vertices and use their neighborhoods to form the connections. In particular, in order to find a $K_n$-subdivision, we need to find a collection of vertex-disjoint paths, one between each pair $V_i$ and $V_j$ for $i<j$. 

Let $\gamma=\eps/11$ and let $S$ denote the empty graph on vertex-set $V_1\cup\ldots\cup V_n$. 
Note that $G'=G[V(G)\setminus W]\sim \U^\ell(K_{n-1})$ and that none of its edges have been revealed so far. Note also that each $V_i$ is a transversal of $G'$. So \Cref{lem:winningconditions} applies and by its parts \ref{joined} and \ref{extendable} we may assume that $G'$ is $m$-joined and that $S$ is $(D,m)$-extendable in $G'$ for $(D,m):=(n^{0.99}, 5\ell\log n)$. Furthermore, if $\ell\le \gamma^3n^2/48$, let $M$ be the matching guaranteed by part \ref{pseudorandomstars} of the lemma, otherwise set $M=\emptyset$.

The edges in $M$ already provide us with a part of the desired collection of paths. On a high level, if $\ell$ is small this partial collection connects almost all the pairs and we only need to connect the remaining at most $\gamma n^2$ pairs $V_i,V_j$. If $\ell$ is big the connection might even be empty, but we have much more space.

Let us now fix an auxiliary matching $M'$ chosen greedily so that $M$ and $M'$ are vertex disjoint and $M \cup M'$ contains an edge between every pair of distinct sets $V_i$ and $V_j$. We note that we do not insist that edges of $M'$ belong to $G'$, rather $M'$ is providing us a template for which pairs of vertices we wish to connect by the paths in our desired collection. Note, also that one can choose such an $M'$ greedily since $|V_i|=n-1$ for all $i$. 
Since $M$ connects all but $\gamma n^2$ pairs $V_i,V_j$ provided $\ell\le \gamma^3n^{2}/48$ we know $M'$ contains at most $\gamma n^2$ edges in this case and it contains $\binom{n}{2}$ edges otherwise.
The following claim makes the above-described strategy precise.
\begin{claim*}
    For each edge $xy=e\in M'$, we may find an $xy$-path $P_e$ in $G'$ of length at most \linebreak $3\lceil \log (2m)/\log (D-1)\rceil$ whose internal vertices are disjoint from $V(S)$, and furthermore, $V(P_e)\cap V(P_{f})=\emptyset$ whenever $e\neq f$.
\end{claim*}
\begin{proof}
Let $e_1,e_2,\ldots, e_t$ be the edges in $M'$. Suppose we are given a collection of paths $P_{e_1},\ldots, P_{e_{i-1}}$ with the desired properties as well as with an additional one that $S\cup P_{e_1} \ldots \cup P_{e_{i-1}}$ is a $(D,m)$-extendable subgraph of $G'$. Note that this assertion holds for $i=1$ with the empty collection of paths, since $S$ is $(D,m)$-extendable. Hence, showing we can extend the collection while maintaining the desired properties will complete the proof.

Suppose first that $\ell\leq \gamma^3n^{2}/48.$
Then, we have $i\leq |M'|\le \gamma n^2$ and $\lceil\log (2m)/\log (D-1)\rceil\leq 3$. Therefore, $S\cup P_{e_1} \cup\ldots \cup P_{e_{i-1}}$ has order at most $$
n^2 + 9\gamma n^2  \leq |V(G')| - 10Dm,
$$ 
since $|V(G')|=\ell(n-1)\geq (1+\eps)n(n-1)$, and $10Dm=50n^{0.99}\log n \cdot \ell\le \gamma\ell n,$ for $n$ large enough. Since vertices in $e_i$ belong to $S$ they do not appear in any of $P_{e_1},\ldots, P_{e_{i-1}}$ so we may apply \Cref{lemma:connecting} to find a desired $P_{e_{i}}$.

If on the other hand $\ell>\gamma^3n^{2}/48$,
then $S\cup P_{e_1} \cup\ldots \cup P_{e_{i-1}}$ has order at most 
$$
n^2 + 3n^2 \lceil \log (2m)/\log (D-1)\rceil\leq n^2\log\ell\leq |V(G')| - 10Dm.
$$ 
So once again, \Cref{lemma:connecting} provides us with our desired $P_{e_{i}}$.
\end{proof}
The union of the paths provided by the claim and our branch vertices gives the desired subdivision of $K_n$.
\end{proof}

\section{Finding subdivisions in small lifts}\label{sec:smalll}
In this section, we prove \Cref{thm:smalll}. We begin with the following more explicit expansion lemma for relatively large sets into arbitrary linear-sized sets. 
\begin{lemma}\label{lem:expansion-explicit}
    Let $0<\beta, \gamma\le \frac12$ and $2\le \ell \le  \beta^2\gamma^2n.$ Let $G \sim \mathcal{U}(K_{n}^\ell)$. Given a partial transversal $U$ of size $|U|\ge  2\beta\sqrt{n\ell}$ and a subset $V \subseteq V(G)$ of size $|V| \ge 5\gamma n\ell$, we have that $|N(U) \cap V| < \gamma n \ell$ with probability at most
    $$ e^{-\gamma^3n|U|/4}.$$
\end{lemma}
\begin{proof}
     We remove from $V$ all vertices belonging to fibers which intersect $V$ in less than $\ceil{2\gamma \ell}$ vertices or intersect $U$. Note that this removes at most $n\cdot 2\gamma \ell+2\beta\sqrt{n\ell}\cdot \ell\le2.5\gamma n\ell \le|V|/2$ vertices from $V$. Denote by $V'$ the set of remaining vertices in $V$, so $|V'|\ge |V|/2$. Let $V_1,\ldots, V_t$ denote the non-empty intersections of $V'$ with the fibers. So, $|V_i| \ge \ceil{2\gamma \ell}$ for all $i\in[t]$. Note also that $|V'|\ge |V|/2> 2\gamma n \ell \implies t \ge 2\gamma n$.  

    By a union bound, the probability that there are at least $|V_i|/2\ge \gamma \ell$ vertices in $V_i$ which do not have a neighbor in $U$ is at most 
    $$2^{|V_i|} \cdot \left(1-\gamma\right)^{|U|}\le 2^{\ell} \cdot e^{-\gamma|U|}\le e^{-\gamma|U|/2},$$
    where in the first inequality we are using the fact that whether a given vertex in $U$ has a neighbor in a specified subset of $V_i$ is independent among all such vertices since all vertices of $U$ belong to different fibers (disjoint from $V_i$). 

    Hence, the probability that there are at least $\gamma t/2$ sets $V_i$ for which $|N(U) \cap V_i|< |V_i|/2$ is at most 
    $$2^{t}\cdot e^{-\gamma^2 t|U|/4}\le e^{-\gamma^2 t|U|/8}\le e^{-\gamma^3 n|U|/4}$$ since different $V_i$'s belong to different fibers (disjoint from those containing $U$) so these events are mutually independent. Otherwise, since any $\gamma t/2$ sets $V_i$ can account for at most $\gamma n \ell/2 \le |V|/10$ vertices of $V$, we have $|N(U) \cap V'| \ge |V|/5 \ge \gamma n \ell$.
\end{proof}

We are now ready to prove the following precise version of the lower bound from \Cref{thm:smalll}. The proof is slightly more involved compared to that of Theorem~\ref{thm:maintheorem}, a key difference being that a significant portion of the vertex-disjoint paths we find between the branch vertices will be of length two (in the proof of Theorem~\ref{thm:maintheorem} all paths were of length at least three). Another important difference is the choice of the location of the branch vertices. In \Cref{thm:maintheorem} we embed the branch vertices in a single fiber. This is completely infeasible here since such vertices by construction never have any common neighbors so can not be joined by paths of length two (which is easily seen to be required to get the precise result we obtain). Instead, we will embed the branch vertices as a partial transversal, although we will at several later stages of the argument remove a few potential outliers. We show that one can connect most of the pairs of our would-be branch vertices with paths of length two through distinct vertices giving us an almost-clique subdivision. Finally, we use the extendability method to connect the remaining pairs through a small collection of fibers we set aside at the beginning for this purpose. 

\begin{theorem}\label{thm:upper-bnd-small-ell-proof}
    Let $0<\eps<2^{-12}$. For $\ell \le \eps^{16}n$, a graph sampled according to $\mathcal{U}(K_{n}^\ell)$ with high probability as $n \to \infty$ contains a topological clique of order at least $(1-43\eps)\sqrt{\frac{2n\ell}{1-1/\ell}}$.
\end{theorem}
\begin{proof}
    Throughout the proof, we will assume that $n$ is large enough for a variety of estimates involving it to hold. The theorem trivially holds for $\ell=1$ so let us assume $\ell \ge 2$.
    Let us fix an arbitrary partial transversal $B$ of size $b=\ceil{(1-2\eps)\sqrt{\frac{2n\ell}{1-1/\ell}}}$ and partition the fibers into two sets $F_1$ and $F_2$ such that $|F_1|=\ceil{(1-\eps)n}, |F_2|=\floor{\eps n}$ and $B$ is contained within fibers in $F_1$.

    Since $B$ is a partial transversal the number of edges $B$ induces follows $\Bin\left({\binom{b}{2},\frac1{\ell}}\right)$ and so is with high probability at least $(1-\eps^2)\binom{b}{2}/{\ell}$. We will need another property of edges between fibers in $F_1$ which holds with high probability as guaranteed by the following claim.

    \begin{claim*}
        With high probability for any collection $P$ of $\ceil{\eps^2 \binom{b}{2}}$ pairs of vertices in $B$ and any set $S$ of $\ceil{\eps n \ell}$ vertices from fibers in $F_1$ there is a pair in $P$ with a common neighbor in $S \setminus B$.
    \end{claim*}
    \begin{proof}
        Let us first fix some $P$ and $S$ as in the statement of the claim and estimate the probability that they fail the condition stated in the claim. 
        It will be convenient to work with an auxiliary graph $H$ with vertex set $B$ and edges being pairs in $P$. 
        Let us remove from $H$ repeatedly any vertex with degree less than half of the average. Since this does not reduce the average degree we are left with a subgraph with minimum degree at least $\eps^2(b-1)/2$. Let us now choose a bipartition of the remaining vertices $B_1 \sqcup B_2$ maximizing the number of edges across. Let $H'\subseteq H$ be the bipartite subgraph consisting only of these edges. In particular, we have minimum degree in $H'$ being at least $\eps^2(b-1)/4$. Suppose $S$ contains at most as many vertices in the fibers containing the vertices of $B_1$ as it does in the fibers containing the vertices in $B_2$.  
        Let us remove from $S$ any vertex which belongs to the same fiber as $B_1$. This leaves at least $\eps n\ell/2$ vertices in $S$.  
        Now for every vertex $v \in B_1$ let $P_v$ be a fixed set of $\ceil{\eps^2(b-1)/4}\ge 2\eps^3 \sqrt{n\ell}$ of its neighbors in $H'$. 
        
        Let us now reveal all the edges between the fibers containing $B_2 \cup S$. Since each $P_v$ is also a partial transversal and $S$ is a set of vertices of size at least $\eps n \ell/2$ we have by \Cref{lem:expansion-explicit} (applied with $\gamma:=\eps/10$ and $\beta:=\eps^3$) that $|N_G(P_v) \cap S |\ge \eps n \ell/10\ge 2\eps^{2} n \ell$ with probability at least $1-e^{-\eps^{7} n\sqrt{n \ell}}$. So by a union bound, we can ensure this happens with high probability for all $v \in B_1$ (without revealing any edges incident to the fibers containing $B_1$). 
        Let us fix an outcome for which this does indeed occur.

        Note that if there is an edge in $G$ between $v$ and $S_v:=N_G(P_v) \cap (S \setminus B)$ this gives a vertex in $S \setminus B$ adjacent to both vertices making an edge in $H'$, so a pair in $P$ and the statement of the claim holds for our choice of $P$ and $S$. Note that $|S_v|=|N_G(P_v) \cap (S \setminus B)|\ge |N_G(P_v) \cap S|-|B| \ge 2\eps^2 n \ell-b \ge \eps^2 n \ell.$
        By splitting $S_v$ according to fibers and applying \Cref{lem: randomness} to each part, the probability that $v$ does not have a neighbor in $S_v$ is at most $e^{-|S_v|/(2\ell)}\le e^{-\eps^2 n/2}$.  
        Since there are at least $2\eps^3 \sqrt{n \ell}$ vertices in $B_1$ (this being a lower bound on the minimum degree in $H'$) and all of them belong to different fibers, making these events mutually independent, the chance this happens for every vertex in $B_1$ is at most $e^{-\eps^{5}n \sqrt{n\ell}}$.

        This shows that the statement of the claim fails for a fixed choice of $P$ and $S$ with probability at most $2e^{-\eps^{7} n\sqrt{n \ell}}.$ 
        There are at most $2^{n\ell}$ choices for both $P$ and $S$.
        So by taking a union bound, we get the probability that the claim fails is at most 
        $$2^{2n\ell}\cdot 2e^{-\eps^{7} n\sqrt{n \ell}}=o(1),$$
        as desired.        
    \end{proof}

    Let us now reveal all edges between fibers in $F_1$ and assume our outcome satisfies both the claim and that $B$ induces at least $(1-\eps^2)\binom{b}{2}/{\ell}$ edges. 
    Using the claim statement repeatedly we can join a new pair of vertices in $B$ (not already joined by an edge) by a path of length two using a unique vertex so long as there are at least $\eps^2\binom{b}{2}$ pairs and at least $\eps n \ell$ vertices from fibers in $F_1$ left. Since $(1-\eps^2)\binom{b}{2}/{\ell}+(1-2\eps)\ell n \ge (1-\eps^2)\binom{b}{2}$ the former condition bottlenecks first. So, we are able to connect all but $\eps^2\binom{b}{2}$ of the pairs of vertices in $B$ by paths of length at most two using only edges between vertices of $B$ or vertices belonging to fibers in $F_1$.
    
    Let us remove from $B$ any vertex which is still missing more than $\eps b/40$ paths. Since we are missing in total at most $\eps^2 \binom{b}{2}$ paths we can remove at most $40\eps b$ vertices and are left with a partial transversal $B'\subseteq B$ with $|B'|=b'\ge (1-40\eps)b$ and a subdivision contained within fibers in $F_1$, with branch vertices being $B'$ and using only paths of length at most two. In addition, every vertex of $B'$ is joined in the current subdivision to all but at most $\eps b/40$ other vertices in $B'$.

    \begin{claim*}
    With high probability, there is a collection of at least $b'-\eps \sqrt{n \ell}$ vertex disjoint stars, each of size $\ceil{\eps b/40}$, centered at a vertex in $B'$, with leaves within fibers in $F_2$, and their union saturating at most half
    of any fiber in $F_2$.    
    \end{claim*}
    \begin{proof}
    Let $V'$ denote an arbitrary subset of vertices in fibers $F_2$ containing $\floor{\ell/2}$ vertices from each fiber. In particular, $|V'|\ge \eps n\ell/4$.
    Let us now reveal the edges between $B$ and fibers in $F_2$. By a union bound over all subsets $U$ of $B'$ of size $\ceil{\eps\sqrt{n\ell}}$ and all subsets $V''$ of $V'$ of size $\ceil{\eps n \ell/8}$ via \Cref{lem:expansion-explicit} (applied with $\beta:=\eps/2$ and $\gamma=\eps/40$) we can ensure that with probability at least 
    $$1-2^b\cdot 2^{\eps n \ell}\cdot e^{-\eps^4 n\sqrt{n \ell}/2^{13}}=1-o(1).$$
    $|N(U) \cap V''| \ge \eps n\ell/40$ for any such $U$ and $V''$. So, in particular, there is a vertex in $U$ with at least $\sqrt{n\ell}/40\ge \ceil{\eps b/40}$ neighbors in $V''$. This allows us to repeatedly find a new star so long as they do not cover at least  $b'-\eps \sqrt{n\ell}$ vertices of $B'$ or more than $\eps n \ell/8$ vertices of $V'$. Since $b' \cdot \ceil{\eps b/40} \le \eps n\ell/8$ the former case occurs first and the claim follows.
    \end{proof}

    Let us fix an outcome in which the second claim also holds (for which we need only reveal the edges between fibers in $F_1$ and $F_2$) and let $L$ denote the union of the leaves of the stars. Note that $|L| \le b'\cdot \ceil{\eps b/40} \le \eps n\ell/8$. 
    We now reveal all edges between fibers in $F_2$, which we denote by $W_1,\ldots, W_t$ (where we know $t \ge \eps n/2$).  
    Since we have not yet revealed any information about these edges they induce a uniformly random lift. The following claim verifies that the empty graph on the vertex set $L$ is extendable in $G[W_1\cup \ldots \cup W_{t}]$, allowing us to start the iterative path embedding strategy. Let $(D,m):=(n^{0.99}, 5\ell\log n).$
     
    \begin{claim*} With high probability, the empty graph on $L$ is a $(D,m)$-extendable subgraph of $G[W_1\cup \ldots \cup W_{t}]$.
    \end{claim*}
    \begin{proof}
        \Cref{lem:expansion2} applied with $G:=G[W_1\cup \ldots \cup W_{t}]$, $\eps:=1/18$ and $V:=(W_1\cup \ldots \cup W_{t})\setminus L$, gives that with high probability, for all $U\subseteq W_1\cup \ldots \cup W_{t}$, such that $|U|\le 10\ell\log n,$ we have $$|N'(U)\cap ((W_1\cup \ldots \cup W_{t})\setminus V(L))|\geq \min\left\{\frac{ \eps n|U|}{72}, \frac{\eps n\ell}{4\cdot 18^6}\right\}-|U|\ge \frac{n}{40\cdot18^6\log n} \cdot |U|-|U|\ge n^{0.99}|U|,$$
     establishing $(D,m)$-extendability.
    \end{proof}
We further know that, $G[W_1 \cup \ldots \cup W_t]$ is with high probability $m$-joined by \Cref{lem: joined}. Let us now fix an outcome of $G[W_1 \cup \ldots \cup W_t]$ in which both this and the above (third) claim occur. We now apply \Cref{lemma:connecting} repeatedly to pairs of vertices in $L$, to obtain a collection of vertex disjoint paths of length at most $3\ceil{\frac{\log (10 \ell \log n)}{\log {(n^{0.99}-1)}}}\le 6$, with exactly one path between any pair of stars whose centers are not already joined by our subdivision. Note that we can choose the endpoints for these paths since our stars have size $\ceil{\eps b/40}$ and each center only needs to be joined to at most $\ceil{\eps b/40}$ other ones. Note also that we can keep applying \Cref{lemma:connecting} since at any point during the embedding process we only use up to $|L|+6 \cdot \eps^2b^2/2\le \eps n \ell/8 + 12\eps^2 n\ell \le \eps n\ell/4 \le |V(G[W_1 \cup \ldots \cup W_t])|-10Dm$ vertices, and that $L$ together with the already embedded paths gives a $(D,m)$-extendable subgraph (with degree of any vertex of $L$ which we still have to connect remaining $0$). These paths connect the remaining pairs of vertices among the centers of our stars giving us a subdivision of a clique with branch vertices being centers of our stars. The order of the clique equals the number of stars which is at least $b'-\eps \sqrt{n\ell}\ge (1-40\eps)b-\eps b\ge (1-43\eps)\sqrt{\frac{2n\ell}{1-1/\ell}}$, as desired.
\end{proof}

\section{Non-existence of subdivisions in small lifts}\label{sec:lwr-bnd}
We begin the section by proving the simpler of our two lower bounds, namely the lower bound part of \Cref{thm:smalll}.

\begin{theorem}\label{thm:general-lower-bnd-proof}
    Let $0<\eps$ and $2 \le \ell$, a graph sampled according to $\mathcal{U}(K_{n}^\ell)$ with high probability as $n \to \infty$ does not contain a subdivision of a clique of order at least $(1+\eps)\sqrt{\frac{2n\ell}{1-1/\ell}}$.
\end{theorem}
\begin{proof}
    By reducing $\eps$ if needed we may assume $\eps<1/4$. Let $b:=\ceil{(1+\eps)\sqrt{\frac{2n\ell}{1-1/\ell}}}$. Note that in any topological clique of order $b$ in $G$ its branch vertices must span at least $\binom{b}{2}+b-n\ell$ edges. Indeed, otherwise for each of the at least $n \ell -b+1$ non-edges we need to have a unique vertex and we have at most $n\ell-b$ available.
    We will show that with high probability any set of $b$ vertices of $G$ spans less than $\binom{b}{2}+b-n\ell$ edges, hence guaranteeing there are no topological cliques of order $b$ in $G$.

    Given a set $B$ of $b$ vertices of $G$ let $B_1,\ldots, B_t$ be its non-empty intersections with different fibers sorted so that $|B_1| \le \ldots \le |B_t|$. Let $X$ denote the random variable counting the number of edges in $G[B]$. Note that since every individual edge exists with probability at most $1/\ell$ we have $\mathbb{E} X \le \binom{b}{2}\cdot \frac1{\ell}$. Let us denote by $X_{ij}$ the random variable counting the number of edges of $G$ between $B_i$ and $B_j$. So, we have $X=\sum_{i<j} X_{ij}$ and $0\le|X_{ij}|\le \min\{|B_i|,|B_j|\}$. Note also that $\sum_{i<j} |B_i|^2 \le \sum_{i<j} |B_i||B_j|\le \binom{b}{2}.$
    So by Azuma's inequality, we have 
    $$\mathbb{P}(X>\mathbb{E}X+ \eps n \ell )\le e^{{-\eps^2 n^2 \ell^2}/{\left(2\binom{b}{2}\right)}}\le e^{-\eps^2n \ell/5}.$$
    On the other hand, there are $\binom{n\ell}{b}$ choices for $B$ so by a union bound the probability that a set of $b$ vertices spanning at least $\mathbb{E}X+ \eps n \ell\le \binom{b}{2}/\ell+\eps n \ell < \binom{b}{2}+b-n\ell$ edges exists is at most
    $$\binom{n\ell}{b}\cdot e^{-\eps^2n \ell}\le e^{3\sqrt{n\ell}\log (n\ell)-\eps^2n \ell/5}=o(1),$$ as desired.
\end{proof}

The rest of the section is dedicated to the proof of \Cref{thm:lower_bound-intro}.
A key part of our proof of this result is a careful multiple exposure argument that reveals our random lift gradually across multiple stages. The most natural way of doing this is to reveal the randomness one matching between fibers at a time, as we have done in the previous sections. This unfortunately is a bit too coarse for our argument here and we will instead need to reveal edges in our lift more gradually. One may view the below argument as revealing the randomness one edge at a time. Note, however, that unlike in the case of the edge exposure martingale in the binomial random graph, in our case, the existence of edges between the same pair of fibers is far from independent. If we are told an edge exists between vertices $uv$, this immediately implies that all other edges incident to $u$ or $v$ between the same pair of fibers do not exist. On the other hand, not having an edge $uv$ implies these other edges all become (slightly) more likely to appear. The following elementary but technical lemma formalizes and extends this intuition in the way we will use in our argument.

\begin{lemma}\label{obs:stepwise revealing}
Suppose $M$ is a uniformly random perfect matching between $A$ and $B$, both of size $\ell$. Fix some disjoint $A_1,A_2\subseteq A$, $B_1,B_2\subseteq B$ and $u\in A\setminus (A_1\cup A_2)$ as well as $U\subseteq B_1$.
We denote by $\mathcal{E}$ the event that $M$ does not contain any edges between $A_1$ and $B_2$ nor between $A_2$ and $B_1$ nor between $u$ and $U$. Let $v\in B_1\setminus U$. Then,
$$
\Pr[uv\in M \mid \mathcal{E}] \geq \frac{\ell-|A_1|-|A_2|-|B_2|}{\ell^2}.
$$
\end{lemma}
\begin{proof}
Let $\mathcal{E}'$ denote the event that $M$ contains no edges between $A_1$ and $B_2$ or $A_2$ and $B_1$. Since $\mathcal{E}\subseteq \mathcal{E}'$ and $\{uv \in M\} \cap \mathcal{E}=\{uv \in M\} \cap \mathcal{E}'$, we have that $\mathbb{P}[uv \in M|\mathcal{E}]\ge \mathbb{P}[uv \in M|\mathcal{E}']$. 

Let $M'\subseteq M$ denote the subset of edges in $M$ that touch $A_2$ or $B_2$. Note that we can determine whether $\mathcal{E}'$ is satisfied just by looking at $M'$. 

Conditioning on $\mathcal{E}'$, the matching $M'$ does not cover $v$ and every vertex in $B_2$ connects to $u$ with probability at most $1/(\ell-|A_1|-|A_2|)$. Thus, $\mathbb{P}[u, v \notin V(M')|\mathcal{E}']\ge \frac{\ell-|A_1|-|A_2|-|B_2|}{\ell}$. Since $\mathcal{E}'$ is determined only by $M'$ and the neighbor of $u$ in $M\setminus M'$ is chosen uniformly from $B\setminus V(M')$, it follows that $\mathbb{P}[uv \in M|u, v\notin V(M'),\mathcal{E}']\ge \frac{1}{\ell}$. Altogether, we obtain
$$\mathbb{P}[uv \in M|\mathcal{E}']\ge \mathbb{P}[uv \in M|u, v\notin V(M'),\mathcal{E}'] \mathbb{P}[u, v \notin V(M')|\mathcal{E}']\ge \frac{\ell-|A_1|-|A_2|-|B_2|}{\ell^2}$$ implying the claim of the lemma.
\end{proof}

We are now ready to prove \Cref{thm:lower_bound-intro}. We state it here in a bit more precise form.
\begin{theorem}\label{thm:lower_bound-proof}
    There exists $C>0$ such that for $\ell\leq n-C\log n$, $G\sim \U^\ell(K_n)$ with high probability as $n \to \infty$ does not contain a subdivision of $K_n$.
\end{theorem}
\begin{proof}
    Let $\eta>0$ be the absolute constant provided by \Cref{lem: constant prob}. We set $C=2^{12}/\eta.$ 
    If $\ell \le 9^5C \log n$, we are done by a result of Drier and Linial \cite[Theorem 4.4]{linial2006minors} which implies there is almost surely no topological clique of order $3\sqrt{\ell n}$. So let us assume $\ell \ge n/9 \ge 9^5C \log n$ in the following.
    
    Our goal is to show that with high probability 
    \begin{enumerate}[label=\textbf{(P)}]
        \item\label{itm:codegree}for every $X\subseteq V(G)$ with $|X|=n$ there exist $u,v\in X$ with at least two common neighbors in $G$ outside of $X$. 
    \end{enumerate}
    Let us first show this implies that $G$ does not contain a subdivision of $K_n$. Indeed, suppose towards a contradiction that we can find one and let $X$ denote the set of its $n$ branch vertices. By our assumption, there exist vertices $u,v\in X$ with at least two common neighbors $a,b$ outside of $X$. Since $G$ is $(n-1)$-regular, all edges incident to a branch vertex must be part of our subdivision. However, this implies paths $uav$ and $ubv$ are both part of the subdivision. Since $a,b$ are not branch vertices this gives two connecting paths between $u$ and $v$ and since their remaining degree is at most $n-3$ neither $u$ nor $v$ can be incident to $n-1$ paths, a contradiction.
    
    It remains to show \ref{itm:codegree} happens with high probability. This will follow from a simple union bound given the following upper bound on the probability that any given $X$ fails \ref{itm:codegree}.
    \begin{claim*}
        For any $X\subseteq V(G)$ with $|X|=n$, the probability $X$ fails \ref{itm:codegree} is at most $n^{-3n}$.
    \end{claim*}
    \begin{proof}
        Let $X_1,\dots,X_n$ denote the intersections of $X$ with the fibers such that $|X_1|\leq\ldots\leq |X_n|$. Let $h$ be maximal such that $|X_{n-h}|\geq \ell/2$. If no such $h$ exists then let $I$ be the empty set. Otherwise, let $I$ denote the set of fibers corresponding to $X_{n-h},\dots,X_n$ and observe that $|I|\leq 2n/\ell$. Let $k$ be minimal such that $\sum_{i=1}^k|X_i|\geq C\log n$. Note that since $|X|=n$ and $|X_n| \le \ell\leq n-C\log n$ we have $k<n$ and hence, 
        \begin{equation}\label{eq:size-bound}
            \sum_{i=1}^k|X_i|< C\log n+|X_k|\le C\log n + \frac{|X_k|+|X_{k+1}|}{2}\le C\log n +\frac{n}{2}.
        \end{equation}
         
        Let $Y$ be an arbitrary subset of $\bigcup_{i=1}^{k}X_i$ of size $C\log n$ and let $J$ be the set of fibers that $Y$ intersects. Note that $|J|\le |Y|=C\log n$. Let us reveal all the edges incident to $Y$ in $G$.
         \Cref{lem:expansion2} applied with $\eps=1/9$ and $V=V(G)$ gives that with high probability, 
        \begin{equation}\label{eq:expansion}
            |N(Y)\setminus X|\geq \min\left\{\frac{Cn\log n}{9},\frac{\ell n}{9^6}\right\}-n\geq \frac{Cn\log n}{10}+(C\log n)^2+\frac{C\log n\cdot 2n }{\ell}.
        \end{equation}
        We note that $|N(Y)\setminus X|$ only depends on the edges incident to $Y$ so to ensure the above we did not need to reveal any edges which are not incident to $Y$.
        
        Since $G$ is $(n-1)$-regular, \eqref{eq:expansion} allows us to greedily find $t:=C \log n /20$ vertex disjoint stars of size $n/20+C\log n+\frac{2n}{\ell}$ with centers $y_1,\ldots, y_t \in Y$ and leaf sets $N_1,\ldots,N_t$ disjoint from $X$.
        Note that since $N_i$ is a subset of the neighborhood of $y_i$, it does not contain two vertices from the same fiber. 
        This allows us to select $N_i'\subseteq N_i$, for each $i$, to be a set of size $n/20$ not containing any vertices from fibers in $I\cup J$. 

        Let $\bigcup_{i=k+1}^nX_i=\{v_1\ldots,v_s\}$. By \eqref{eq:size-bound}, we have $s\geq n/2-C\log n$. For $1\leq i\leq n/64$ and $1\leq j\leq t$, denote by $\mathcal{Y}_{i,j}$ the event that $v_i$ has at least $2$ neighbors in $N'_j$. Note that if any $\Y_{i,j}$ occurs, then $X$ satisfies \ref{itm:codegree} since $v_i$ and $y_j$ have two common neighbors in $N_j'$ so in particular outside of $X$. Hence, the probability that $X$ fails \ref{itm:codegree} is at most 
        \begin{equation}\label{eq:conditioning}
\Pr\Big[\bigcap_{i,j}\:\overline{\Y_{i,j}}\Big]=\prod_{i,j}\Pr\Big[\overline{\Y_{i,j}}\mid \bigcap_{(i',j')<(i,j)}\overline{\Y_{i',j'}}\Big],
        \end{equation}
        where $(i',j')<(i,j)$ refers to the lexicographic ordering, namely either $i'<i$, or $i'=i$ and $j'<j$. 
        
        Now, let us fix some $1\leq i\leq n/64$ and $1\leq j\leq t$. Next, we prove an upper bound on $\Pr\left[\overline{\Y_{i,j}}\mid\bigcap_{(i',j')<(i,j)}\overline{\Y_{i',j'}}\right].$ First, observe that none of the vertices in $\bigcup_{j=1}^{t} N_{j}' \cup\{v_1,\dots,v_s\}$ belong to fibers in $J$. Therefore, we have revealed nothing about the edges between the fibers these vertices belong to when revealing edges incident to $Y$. To prove the upper bound, we partially reveal the edges of the graph $G$. The revealed edges will be such that it is clear if $\bigcap_{(i',j')<(i,j)}\overline{\Y_{i',j'}}$ happens or not. We then prove that no matter how exactly the revealed edges look, as long as $\bigcap_{(i',j')<(i,j)}\overline{\Y_{i',j'}}$ is satisfied, $\overline{\Y_{i,j}}$ happens with probability at most $1-\eta$.
        
        Suppose we reveal all the edges between $v_{i'}$ and $N'_j$ for all $i'<i$ and $1\leq j\leq t$ as well as all the edges between $v_{i}$ and $\bigcup_{j'=1}^{j-1} N'_{j'}$. This information suffices to determine whether the event $\bigcap_{(i',j')<(i,j)}\overline{\Y_{i',j'}}$ occurs or not. In case it does, we have revealed the existence of at most $i-1$ edges incident to $N_j'$ (at most one to each of the vertices $v_1,\ldots, v_{i-1}$) and at most $j-1$ edges incident to $v_i$ (as well as non-existence of a number of edges between $\{v_1,\ldots, v_{i}\}$ and $\bigcup_{j'=1}^{t} N'_{j'}$). 
        Then, there is a set $N\subseteq N'_j$ of at least $n/20-(i+j-2)-1\geq n/32$ vertices which have not yet had any neighbor revealed, nor are they in the same fiber as $v_i$. Recall that $N \subseteq N'_j$ contains at most one vertex per fiber. Therefore, each $v\in N$ is adjacent to $v_i$ independently of whether other vertices in $N$ are adjacent to $v_i$. We denote by $A,B$ the fiber of $v_i$ respectively $v$. Let us analyze what we have revealed so far about the edges between $A$ and $B$. First of all, let $M$ denote the partial matching which has already been revealed between $A$ and $B$. Set $A'=A\setminus V(M)$ and $B'=B\setminus V(M)$. All edges in $M$ are incident to $\left(\bigcup_{j'=1}^{t} N'_{j'}\right)\cap (A\cup B)$. Set $\ell':=|A'|=|B'|\geq \ell-2t$ as well as,
        \begin{gather*}
            A_1 := A'\cap \bigcup_{j'=1}^{t} N'_{j'},\;\;\;\;\;\;
            B_1 := B'\cap \bigcup_{j'=1}^{t} N'_{j'},\;\;\;\;\;\;U := B'\cap \bigcup_{j'=1}^{j-1} N'_{j'}\\
            A_2 := A'\cap \{v_1,\ldots,v_{i-1}\},\;\;\;\;\;\;
            B_2 := F'_{v_i}\cap  \{v_1,\ldots,v_{i-1}\}.
        \end{gather*}
        Note that the remaining matching is a uniformly random perfect matching between $A'$ and $B'$ under the condition that $v\in B_1\setminus U$ and there are no edges between $A_1$ and $B_2$, $A_2$ and $B_1$ as well as $v_i$ and $U$. These are exactly the properties of $\mathcal{E}$ from \Cref{obs:stepwise revealing} with $u=v_i$. Additionally, we have $|A_1|\leq t$ since each $N'_{j'}$ intersects each fiber in at most one vertex, $|A_2|<i\le n/64 \le \ell/7$ and $|B_2|\leq \ell/2$ as $B$ is not a fiber in $I$. Therefore, no matter which exact edges we revealed, we always get by \Cref{obs:stepwise revealing} that
        $$
        \Pr\Big [v_iv\in E(G) \mid \bigcap_{(i',j')<(i,j)}\overline{\Y_{i',j'}}\Big]\geq \frac{\ell'-t-\ell/7-\ell/2}{\ell'^2}=\frac{(1-1/7-1/2-o(1))\ell}{\ell^2}\geq\frac{1}{3\ell}.
        $$
        
        By \Cref{lem: constant prob}, we get $$\Pr\Big[\overline{\Y_{i,j}}\mid\bigcap_{(i',j')<(i,j)}\overline{\Y_{i',j'}}\Big]\leq 1-\eta\leq e^{-\eta}.$$ 
        Combining with \eqref{eq:conditioning}, it follows that the probability $X$ fails \ref{itm:codegree} is at most
        $$e^{-\eta \cdot \frac{n}{64}\cdot \frac{C \log n}{20}}\le e^{-3n \log n}=n^{-3n},$$
        as desired.
    \end{proof}
    Since there are at most $(\ell n)^n \le n^{2n}$ choices for $X$ we can ensure via a union bound combined with the above claim that \ref{itm:codegree} holds with high probability for all $X$. This completes the proof as argued above.
\end{proof}

\section{Concluding remarks and open problems}\label{sec:conc-remarks}
In this paper, settling a conjecture of Drier and Linial, we determined (up to lower order terms) the threshold in terms of how large we need to take $\ell$ so that almost all $\ell$-lifts of $K_n$ contain a topological clique of order $n$. Perhaps the most immediate next question is if one can determine the threshold even more precisely.

In addition, our results determine the Haj\'os number of a typical $\ell$-lift of $K_n$ precisely when $\ell \ge (1+o(1))n$ and up to lower order terms when $\ell \ll n$. This leads to the natural question of what happens in the remaining regime when $\Omega(n) \le \ell \le (1-o(1))n,$ when we only know the answer up to a constant factor. Since our results show the answer is approximately $\sqrt{2n\ell}$ when $\ell=\Omega(n)$ and approximately $\sqrt{n\ell}$ when $\ell\ge (1-o(1))n$ there is likely another gradual transition occurring in this regime, similarly as the answer transitions from being approximately $2\sqrt{n\ell}$ for $\ell=2$ to $\sqrt{2n\ell}$ when $\ell \gg 1$. Our tools and ideas should be useful here although a much more precise version of \Cref{thm:upper-bnd-small-ell-proof} would be needed as well as an improved version of \Cref{thm:general-lower-bnd-proof}. 

Another interesting further direction is what happens for a typical lift of graphs other than complete ones. Here, Witkowski \cite{witkowski2013} determined the answer precisely provided $\ell$ is sufficiently large compared to the order of the graph being lifted. 
The ideas behind our argument extend to many other graphs and can be used to improve the lower bound requirement on $\ell$.

Let us also highlight another very interesting question raised by Drier and Linial \cite{linial2006minors}. They asked if any lift of $K_n$ can have Hadwiger number $o(n)$. Here, the Hadwiger number of a graph $G$ is the largest $m$ for which we can find a minor of $K_m$ in $G$. The only difference to the Haj\'os number is that one is looking for a minor rather than a topological minor.  The typical case is well understood (see \cite{linial2006minors}) and for $\ell \gg \log n$ one can find a clique minor of order which grows faster than just linear in $n$ showing the answers are very different between the two notions.

Let us also mention a very nice question of Fountoulakis, K\"{u}hn, and Osthus \cite{top-minors-expanding-graphs}. To state it we first need a definition. An $n$-vertex graph is said to be an $(\alpha,t)$-expander if any $X \subseteq V(G)$ with $|X| \le \alpha n/t$ satisfies $|N(X)| \ge (t+1)|X|$. They asked for which values of $\alpha,t,d$ does a $d$ regular \emph{$(\alpha,t)$-expander} necessarily contain a topological clique of order $d+1$. Note that \Cref{lem:expansion2} with $\eps=1/9$ guarantees that a typical $\ell$-lift of $K_n$ is an $(n/10,9^{6})$-expander (and with more care, one can establish even stronger expansion properties, see e.g.\ \cite{spectrum-random-lifts} for the case when $\ell$ is relatively small compared to $n$). Hence, our \Cref{thm:lower_bound-intro} gives a negative answer to this question even for such extremely good expanders. We note that the result of Drier and Linial mentioned in the introduction already establishes this for $\ell=o(n)$ and that Dragani\'{c}, Krivelevich, and Nenadov \cite{draganic2022rolling} have a very different construction of good $d$-regular expanders without a topological clique of order $d+1$.
On a more positive side, our main result \Cref{thm:maintheorem-intro} shows that for $\ell\ge (1+o(1))n$ a typical $\ell$-lift of $K_n$ does indeed satisfy this property. It would be very interesting if one could identify a simple expansion property which suffices to guarantee it and captures our result. Unfortunately, our arguments do rely in several places on the specific structure of lift graphs. In this direction, there is an approximate result, already mentioned in the introduction, due to Dragani\'{c}, Krivelevich, and Nenadov \cite{draganic2022rolling} who showed that in an $(n,d,\lambda)$-graph with $240\lambda < d \le n^{1/5}/2$, one can find a topological clique of order $d-O(\lambda)$.  

Our argument can be made algorithmic in the sense that one can devise a polynomial time algorithm which either finds a $K_n$ minor or outputs a certificate that the lift we work with is ``atypical''. Indeed, as described in Section~\ref{sec:uppr-bnd}, our algorithm for finding a $K_n$-subdivision consists of three steps, where in the first step we embed $n$ vertex disjoint stars with $n-1$ leaves, and in the second step we find a maximal matching between the vertices of the leaves of the stars connecting as many distinct pairs of stars as possible. Both of these steps can easily be performed via polynomial time algorithms. In the third step, we rely on the extendability method (\Cref{lemma:connecting}) to find several vertex-disjoint paths inside a good expander. To perform this step efficiently, we can use an algorithmic version of \Cref{lemma:connecting} from recent work of Dragani\'c, Krivelevich, and Nenadov \cite{draganic2022rolling} (see Theorem 3.5 in \cite{draganic2022rolling}).

\textbf{Acknowledgments.} We would like to thank Noga Alon for useful discussions and Michael Krivelevich for useful comments. The first author would like to gratefully acknowledge the support of the Oswald Veblen Fund. The third would like to thank Princeton University Department of Mathematics for its hospitality for the duration he was hosted as a Visiting Student Research Collaborator as this research was taking place.


\vspace{-0.3cm}


\begin{thebibliography}{10}

\bibitem{ajtai-komlos-szemeredi}
M.~Ajtai, J.~Koml\'os, and E.~Szemer\'edi.
\newblock Topological complete subgraphs in random graphs.
\newblock {\em Studia Sci. Math. Hungar.}, 14(1-3):293--297, 1979.

\bibitem{linial1999}
A.~Amit and N.~Linial.
\newblock Random graph coverings. {I}. {G}eneral theory and graph connectivity.
\newblock {\em Combinatorica}, 22(1):1--18, 2002.

\bibitem{edge-expansion}
A.~Amit and N.~Linial.
\newblock Random lifts of graphs: edge expansion.
\newblock {\em Combin. Probab. Comput.}, 15(3):317--332, 2006.

\bibitem{alpha-chi}
A.~Amit, N.~Linial, and J.~Matou{\v{s}}ek.
\newblock Random lifts of graphs: independence and chromatic number.
\newblock {\em Random Structures Algorithms}, 20(1):1--22, 2002.

\bibitem{bernshteyn2023dp}
A.~Bernshteyn, D.~Dominik, H.~Kaul, and J.~A. Mudrock.
\newblock {DP}-{C}oloring of {G}raphs from {R}andom {C}overs.
\newblock {\em arXiv preprint 2308.13742}, 2023.

\bibitem{bilu-linial}
Y.~Bilu and N.~Linial.
\newblock Lifts, discrepancy and nearly optimal spectral gap.
\newblock {\em Combinatorica}, 26(5):495--519, 2006.

\bibitem{chi-Gnp}
B.~Bollob\'as.
\newblock The chromatic number of random graphs.
\newblock {\em Combinatorica}, 8(1):49--55, 1988.

\bibitem{bollobas-catlin}
B.~Bollob\'as and P.~A. Catlin.
\newblock Topological cliques of random graphs.
\newblock {\em J. Combin. Theory Ser. B}, 30(2):224--227, 1981.

\bibitem{bollobas-thomason}
B.~Bollob\'as and A.~Thomason.
\newblock Proof of a conjecture of {M}ader, {E}rd{\H{o}}s and {H}ajnal on topological complete subgraphs.
\newblock {\em European J. Combin.}, 19(8):883--887, 1998.

\bibitem{eigenvalues}
C.~Bordenave and B.~Collins.
\newblock Eigenvalues of random lifts and polynomials of random permutation matrices.
\newblock {\em Ann. of Math. (2)}, 190(3):811--875, 2019.

\bibitem{brailovskaya2022universality}
T.~Brailovskaya and R.~van Handel.
\newblock Universality and sharp matrix concentration inequalities.
\newblock {\em arXiv preprint 2201.05142}, 2022.

\bibitem{HCs2}
K.~Burgin, P.~Chebolu, C.~Cooper, and A.~M. Frieze.
\newblock Hamilton cycles in random lifts of graphs.
\newblock {\em European J. Combin.}, 27(8):1282--1293, 2006.

\bibitem{catlin}
P.~A. Catlin.
\newblock Haj\'os' graph-coloring conjecture: variations and counterexamples.
\newblock {\em J. Combin. Theory Ser. B}, 26(2):268--274, 1979.

\bibitem{draganic2022rolling}
N.~Dragani\'{c}, M.~Krivelevich, and R.~Nenadov.
\newblock Rolling backwards can move you forward: on embedding problems in sparse expanders.
\newblock {\em Trans. Amer. Math. Soc.}, 375(7):5195--5216, 2022.

\bibitem{draganic2024hamiltonicity}
N.~Dragani{\'c}, R.~Montgomery, D.~M. Correia, A.~Pokrovskiy, and B.~Sudakov.
\newblock Hamiltonicity of expanders: optimal bounds and applications.
\newblock {\em arXiv preprint 2402.06603}, 2024.

\bibitem{linial2006minors}
Y.~Drier and N.~Linial.
\newblock Minors in lifts of graphs.
\newblock {\em Random Structures Algorithms}, 29(2):208--225, 2006.

\bibitem{Hajos-random-graphs}
P.~Erd{\H{o}}s and S.~Fajtlowicz.
\newblock On the conjecture of {H}aj\'os.
\newblock {\em Combinatorica}, 1(2):141--143, 1981.

\bibitem{top-minors-expanding-graphs}
N.~Fountoulakis, D.~K\"uhn, and D.~Osthus.
\newblock Minors in random regular graphs.
\newblock {\em Random Structures Algorithms}, 35(4):444--463, 2009.

\bibitem{hajos-summary}
J.~Fox, C.~Lee, and B.~Sudakov.
\newblock Chromatic number, clique subdivisions, and the conjectures of {H}aj{\'o}s and {E}rd{\H{o}}s-{F}ajtlowicz.
\newblock {\em Combinatorica}, 33(2):181--197, 2013.

\bibitem{distributed}
M.~G\"o\"os and J.~Suomela.
\newblock Locally checkable proofs in distributed computing.
\newblock {\em Theory Comput.}, 12:P19, 33, 2016.

\bibitem{PMs}
C.~Greenhill, S.~Janson, and A.~Ruci\'nski.
\newblock On the number of perfect matchings in random lifts.
\newblock {\em Combin. Probab. Comput.}, 19(5-6):791--817, 2010.

\bibitem{unique-games}
J.~A. Grochow and J.~Tucker-Foltz.
\newblock Computational topology and the unique games conjecture.
\newblock In {\em 34th {I}nternational {S}ymposium on {C}omputational {G}eometry}, volume~99 of {\em LIPIcs. Leibniz Int. Proc. Inform.}, pages Art. No. 43, 16. Schloss Dagstuhl. Leibniz-Zent. Inform., Wadern, 2018.

\bibitem{expander-survey}
S.~Hoory, N.~Linial, and A.~Wigderson.
\newblock Expander graphs and their applications.
\newblock {\em Bull. Amer. Math. Soc. (N.S.)}, 43(4):439--561, 2006.

\bibitem{hyde2023spanning}
J.~Hyde, N.~Morrison, A.~M{\"u}yesser, and M.~Pavez-Sign{\'e}.
\newblock Spanning trees in pseudorandom graphs via sorting networks.
\newblock {\em arXiv preprint 2311.03185}, 2023.

\bibitem{crux}
S.~Im, J.~Kim, Y.~Kim, and H.~Liu.
\newblock Crux, space constraints and subdivisions.
\newblock {\em to appear in J. Combin. Theory Ser. B, arXiv preprint 2207.06653}, 2022.

\bibitem{quantum}
F.~G. Jeronimo, T.~Mittal, R.~O'Donnell, P.~Paredes, and M.~Tulsiani.
\newblock Explicit abelian lifts and quantum {LDPC} codes.
\newblock In {\em 13th {I}nnovations in {T}heoretical {C}omputer {S}cience {C}onference}, volume 215 of {\em LIPIcs. Leibniz Int. Proc. Inform.}, pages Art. No. 88, 21. Schloss Dagstuhl. Leibniz-Zent. Inform., Wadern, 2022.

\bibitem{komlos-szemeredi}
J.~Koml\'os and E.~Szemer\'edi.
\newblock Topological cliques in graphs. {II}.
\newblock {\em Combin. Probab. Comput.}, 5(1):79--90, 1996.

\bibitem{kuhn-osthus-girth}
D.~K\"uhn and D.~Osthus.
\newblock Topological minors in graphs of large girth.
\newblock {\em J. Combin. Theory Ser. B}, 86(2):364--380, 2002.

\bibitem{improved-kuhn-osthus}
D.~K\"uhn and D.~Osthus.
\newblock Improved bounds for topological cliques in graphs of large girth.
\newblock {\em SIAM J. Discrete Math.}, 20(1):62--78, 2006.

\bibitem{PMs2}
N.~Linial and E.~Rozenman.
\newblock Random lifts of graphs: perfect matchings.
\newblock {\em Combinatorica}, 25(4):407--424, 2005.

\bibitem{mader-resolution}
H.~Liu and R.~Montgomery.
\newblock A proof of {M}ader's conjecture on large clique subdivisions in {$C_4$}-free graphs.
\newblock {\em J. Lond. Math. Soc. (2)}, 95(1):203--222, 2017.

\bibitem{spectrum-random-lifts}
E.~Lubetzky, B.~Sudakov, and V.~Vu.
\newblock Spectra of lifted {R}amanujan graphs.
\newblock {\em Adv. Math.}, 227(4):1612--1645, 2011.

\bibitem{HCs}
T.~{\L}uczak, {\L}.~Witkowski, and M.~Witkowski.
\newblock Hamilton cycles in random lifts of graphs.
\newblock {\em European J. Combin.}, 49:105--116, 2015.

\bibitem{coding}
X.~Ma and E.-H. Yang.
\newblock Constructing {LDPC} {C}odes by 2-{L}ifts.
\newblock In {\em 2007 IEEE International Symposium on Information Theory}, pages 1231--1235, 2007.

\bibitem{mader}
W.~Mader.
\newblock An extremal problem for subdivisions of {$K^-_5$}.
\newblock {\em J. Graph Theory}, 30(4):261--276, 1999.

\bibitem{interlacing}
A.~W. Marcus, D.~A. Spielman, and N.~Srivastava.
\newblock Interlacing families {I}: {B}ipartite {R}amanujan graphs of all degrees.
\newblock {\em Ann. of Math. (2)}, 182(1):307--325, 2015.

\bibitem{montgomery2019spanning}
R.~Montgomery.
\newblock Spanning trees in random graphs.
\newblock {\em Adv. Math.}, 356:106793, 92, 2019.

\bibitem{nir2021chromatic}
J.~Nir and X.~P. Gim{\'e}nez.
\newblock The chromatic number of random lifts of complete graphs.
\newblock {\em arXiv preprint 2109.13347}, 2021.

\bibitem{near-ramanujan}
R.~O'Donnell and X.~Wu.
\newblock Explicit near-fully {X}-{R}amanujan graphs.
\newblock In {\em 2020 {IEEE} 61st {A}nnual {S}ymposium on {F}oundations of {C}omputer {S}cience}, pages 1045--1056. IEEE Computer Soc., Los Alamitos, CA, 2020.

\bibitem{paul-survey}
P.~Seymour.
\newblock Hadwiger's conjecture.
\newblock In {\em Open problems in mathematics}, pages 417--437. Springer, 2016.

\bibitem{thomassen}
C.~Thomassen.
\newblock Some remarks on {H}aj\'os' conjecture.
\newblock {\em J. Combin. Theory Ser. B}, 93(1):95--105, 2005.

\bibitem{witkowski2013}
M.~Witkowski.
\newblock Random lifts of graphs are highly connected.
\newblock {\em Electron. J. Combin.}, 20(2):P23, 11, 2013.

\end{thebibliography}
\end{document}